\documentclass[10pt]{amsart}
\usepackage{amssymb}
\usepackage{amsmath}
\usepackage{amsfonts}
\usepackage{mathrsfs}
\usepackage{psfrag}
\usepackage{verbatim}
\usepackage{color}
\usepackage{overpic}
\usepackage{tikz}
\usepackage{bm}
\usepackage{bbm}
\usepackage{dsfont}
\usepackage{hyperref}
\usepackage{setspace}

\hypersetup{
	linkcolor=red,          
	citecolor=green,        
	filecolor=magenta,      
	urlcolor=cyan           
}

\DeclareGraphicsExtensions{}
\theoremstyle{plain}

\theoremstyle{definition}

\theoremstyle{remark}

\newcommand{\eps}{\varepsilon}

\newcommand{\kap}{\varkappa}

\newcommand{\BR}{\mathbb{R}}
\newcommand{\BP}{\mathbb{P}}
\newcommand{\BE}{\mathbb{E}}
\newcommand{\BZ}{\mathbb{Z}}

\newcommand{\filt}{\mathscr{F}}

\newcommand{\sfG}{\mathsf{G}}

\newcommand{\cc}{\mathsf{c}}

\newcommand{\ind}{\mathds{1}}

\newtheorem{rmk}{Remark}
\newtheorem{thm}{Theorem}
\newtheorem{lem}[thm]{Lemma}

\newtheorem{pro}[thm]{Proposition}

\begin{document}
\title[Approximation of noisy control systems with sampling]{Approximation of linear controlled dynamical systems with small random noise and fast periodic sampling}
\author{Shivam Dhama}
\email{shivam.dhama@iitgn.ac.in} 

\author{Chetan D. Pahlajani}
\email{cdpahlajani@iitgn.ac.in} 
\thanks{The second author acknowledges research support from DST SERB Project No. EMR/2015/000904.}
\maketitle
\address{Discipline of Mathematics, Indian Institute of Technology Gandhinagar, Palaj, Gandhinagar 382355, India}

\begin{abstract}
In this paper, we study the dynamics of a linear control system with given state feedback control law in the presence of fast periodic sampling at temporal frequency $1/\delta$ ($0 < \delta \ll 1$), together with small white noise perturbations of size $\eps$ ($0<\eps \ll 1$) in the state dynamics. For the ensuing continuous-time stochastic process indexed by two small parameters $\eps,\delta$, we obtain
 effective ordinary and stochastic differential equations  describing the mean behavior and the typical fluctuations about the mean in the limit as $\eps,\delta \searrow 0$. The effective fluctuation process is found to vary, depending on whether $\delta \searrow 0$ faster than/at the same rate as/slower than $\eps \searrow 0$. The most interesting case is found to be the one where $\delta,\eps$ are comparable in size; here, the limiting stochastic differential equation for the fluctuations has both a diffusive term due to the small noise and an effective drift term which captures the cumulative effect of the fast sampling. In this regime, our results yield a time-inhomogeneous Markov process which provides a strong (pathwise) approximation of the original non-Markovian process, together with estimates on the ensuing error. A simple example involving an infinite time horizon linear quadratic regulation problem illustrates the results. 
\end{abstract}

\section{Introduction}

The control of dynamical systems governed by ordinary differential equations ({\sc ode}) frequently involves situations where control actions are computed and executed at discrete time instants, rather than being updated continuously. An example is the sample-and-hold implementation of a continuous-time state feedback control law, where the state of the plant is sampled at discrete time instants and used to compute the control action, which is then held fixed until the next sample is taken. Such \textit{sampled-data systems} with interweaving of continuous and discrete dynamics arise naturally in the context of computer-controlled systems \cite{ChenFrancis-book,YuzGoodwin-book} and networked control systems \cite{antunes2011volterra,antunes2013stability,ModlingandAnalysisNCS14} and have stimulated much research in the field of \textit{hybrid dynamical systems} \cite{GST-HybDynSys-book}. Much of the research effort centers on investigating whether, and to what extent, the system with sampling retains various important properties (e.g., stability) of its idealized counterpart, viz., the fully continuous-time system with continuous state measurements and control updates.

Dynamical systems of interest---with or without control---are almost always subject to uncertainties, either due to imperfect modelling or external disturbances or both. This motivates the use of stochastic process models, which, in the continuous-time case with white noise perturbations, take the form of \textit{stochastic differential equations} ({\sc sde}) \cite{Oksendal,KS91}. For stochastic processes solving {\sc sde} with small noise and/or a separation of time scales, asymptotic techniques are invaluable in providing tractable tools to understand system dynamics, whether through obtaining simpler approximate models or by computing rates of decay of probabilities of rare events. Such \textit{limit theorems} for stochastic processes have been extensively studied; see, for instance, \cite{DZ,FW_RPDS,SHS-RPM,KKO-book}.

In the present work, we obtain strong (pathwise) approximations of control systems with periodic sampling and random perturbations in the \textit{small noise, fast sampling} limit. We focus on a sample-and-hold implementation of a given state feedback control law in a linear system in the presence of periodic sampling at temporal frequency $1/\delta$ ($0 < \delta \ll 1$) and white noise perturbations of size $\eps$ ($0<\eps \ll 1$) in both the state dynamics and the state measurements. The dynamics of the state variable are now given by a continuous-time stochastic process indexed by two small parameters $\eps,\delta$. For this process, we compute something like a first-order perturbation expansion in terms of effective (i.e., independent of $\delta,\eps$) {\sc ode} and {\sc sde}, which describe, respectively, the \textit{mean} behavior and the \textit{typical fluctuations} about the mean in the limit as $\eps,\delta \searrow 0$. 

While the {\sc ode} for the mean behavior is, not surprisingly, the closed-loop {\sc ode} for continuous sampling and control updates, regardless of how $\eps,\delta \searrow 0$, the exact form of the fluctuation process is found to depend on the relative rates at which $\eps,\delta \searrow 0$. Following \cite{FreidlinSowers,KS2014}, which study stochastic processes with multiple small parameters,\footnote{Note that in \cite{FreidlinSowers}, the parameter $\delta$ relates to homogenization rather than sampling, while in \cite{KS2014}, it corresponds to a fast time scale.} we consider three regimes, depending on whether $\delta \searrow 0$ faster than/at the same rate as/slower than $\eps \searrow 0$. The most interesting case is found to be the one where $\delta,\eps$ are comparable in size (e.g., $\delta/\eps$ is a positive constant); here, the limiting {\sc sde} for the fluctuations has both a diffusive term due to the small noise and an \textit{effective drift} term which captures the cumulative effect of the fast sampling. Our main contribution in this work is to obtain, for each of these regimes, moment estimates on the pathwise error between the true process and its first-order approximation; the latter includes the effects of the {\sc ode} for the mean behavior together with the {\sc sde} for fluctuations.\footnote{Strictly speaking, for the scaling we consider, the first-order correction is given by an {\sc sde} in two out of three regimes, viz., when the time $\delta$ between samples goes to zero faster than/at the same rate as the parameter $\eps$ characterizing the small noise. When $\delta$ goes to zero slower than $\eps$, the first-order correction is deterministic and is given by an {\sc ode}.} The results here, which can be thought of as functional analogues of the Law of Large Numbers and the Central Limit Theorem, respectively, thus provide an analytically simpler ``non-hybrid" surrogate for the system of interest, together with estimates on the ensuing error. 

The literature on control systems with sampling is extensive; here, we very briefly touch upon a few facets of this large body of research to provide a bit of context for our work. Much of the work on sampled-data systems takes one of two different approaches. One approach involves obtaining discrete-time models first and then designing the control laws, e.g.,  \cite{NesicTeelKokotovic-SCL1999,NesicTeel-TAC2004}. The other so-called \textit{emulation} approach involves designing control laws in continuous time and then discretizing them using a sample-and-hold device, e.g., \cite{NesicTeelCarnevale-TAC2009,Khalil-TAC-2004,LNT-EurJControl2002,KK-IJRNC-2009}. Important questions in this latter context include quantifying, for a control law which suitably stabilizes the continuous-time system, conditions (e.g., maximum allowable sampling period) which ensure stability of the sampled-data system. Our calculations in this paper are in the spirit of the emulation approach. Motivated by applications in networked control systems, where the components of a control system may be spatially distributed and interact with each other over networks with limited communication capabilities, \textit{periodic event-triggered control} strategies have recently received much attention, see, e.g., \cite{heemels2013periodic,borgers2018time,PETCNNS20}; here, samples are taken periodically and, at each sampling time, ``event-triggering conditions" are checked to decide whether or not transmission of new state/output measurements and control signals takes place. Questions of interest center around exploration of event-triggering conditions which ensure stability or fulfilment of performance criteria, while also satisfying various system and resource constraints.

Various aspects of stochasticity in control systems with sampling have also been explored. These include problems (motivated in part by applications in networked control systems) where samples are taken at random times, e.g., at the arrival times of a Poisson/renewal process \cite{ACM,HespTeel-17thMTNS,TanwaniChatterjeeLiberzon,TanwaniYufereva-Aut2020}, and also sampled control systems with Brownian noise given by Ito {\sc sde}. A few representative papers pertaining to the latter include \cite{anderson2015self} which studies second moment stability of controlled {\sc sde} with self-triggered sampling, \cite{dong2020asymptotic} which addresses stabilization for periodic {\sc sde} with state observations according to time-varying observation interval sequences, \cite{chen2021mean} which considers mean square exponential stability for systems given by Ito {\sc sde} with aperiodic sampling and time delays, \cite{TanwaniYufereva-Aut2020} which focusses on performance analysis of filters for continuous-time nonlinear stochastic systems with the state evolving according to an Ito {\sc sde} and noisy measurements arriving according to the event times of a Poisson process, \cite{dragan2020stochastic} which deals with stochastic linear quadratic optimal control for an Ito {\sc sde} with piecewise-constant admissible controls.



The overwhelming majority of investigations in sampled-data systems address questions of stability, and thus focus on the \textit{limiting} behavior of the system in the large time limit. In contrast, questions regarding the \textit{transient} behavior of the system, e.g., over finite time horizons, how much do sampling and random noise cause a control system to deviate from its noiseless fully continuous-time counterpart, have received much less attention. The importance of such questions about the transient behavior stems from wanting to ensure that the state does not experience undesirable large excursions before the large time behavior starts to set in. A natural follow-up question is whether one can refine the deterministic picture and include---to leading order---the cumulative effect of noise and sampling, thereby getting a more accurate picture of the transient behavior. Our calculations take a step in this direction by obtaining, for an asymptotic problem with small Brownian noise and fast periodic sampling, estimates on the mean of the maximum deviation between the true dynamics and both its zeroth- and first-order perturbation expansions over any finite time interval. 


Mathematically, the calculations here fit into the general setting of asymptotic analysis of stochastic processes with multiple small parameters, where the relative rates at which the parameters vanish determine different asymptotic regimes with differing limiting dynamics. A few representative studies on such limit theorems for stochastic processes include the interplay between large deviations and homogenization \cite{FreidlinSowers}, small noise and multiple time scales \cite{Spil-AppMathOptim,KS2014,athreya2019simultaneous} including the case of fractional Brownian noise \cite{bourguin2020typical}, stochastic averaging with multiple scales \cite{rockner2021diffusion}, small noise and control magnitude for controlled diffusion processes \cite{AriAnupBor18}. Our work is  perhaps closest to \cite{KS2014}, which focuses on the limiting analysis of fluctuations, around the deterministic homogenized limit, for the slow component in a system of small noise {\sc sde} with multiple time scales. There, the forms of both the limiting {\sc ode} for the homogenized limit and the {\sc sde} for the fluctuations are found to vary depending on the exact relation between the fast time scale and the size of the noise; moreover, the limiting fluctuations include additional effective drift terms in certain cases. In the present work, we have a single vector {\sc sde} rather than a slow-fast system, and the focus in our problem is on the interaction between small noise and fast sampling, with the effect of the latter being not fast convergence to the invariant measure but rather the introduction of vanishing \textit{short memory}.\footnote{The memory arises from the fact that the dynamics in between samples does in fact depend on the measurement at the most recent sample.}

The rest of the paper is organized as follows. First, in Section \ref{S:ProblemStatementResults}, we formulate our problem and state the main results: Theorem \ref{T:flln} which describes the mean behavior, and Theorems \ref{T:fluctuations-R-1-2} and \ref{T:fluctuation_reg_3} which together describe fluctuations about the mean for all three regimes. Proofs are provided in Sections \ref{S:FLLN} and \ref{S:Fluct-R12}.
 Next, we illustrate our results on a specific example in Section \ref{S:Example}, and then end with some concluding remarks and directions for future work in Section \ref{S:Conclusions}.

\section{Problem Statement and Results}\label{S:ProblemStatementResults}

In this section, we clearly formulate our problem of interest and state our main results: Theorems \ref{T:flln}, \ref{T:fluctuations-R-1-2} and \ref{T:fluctuation_reg_3}. Fix positive integers $n,m$. The state and control spaces for our problem will be $\BR^n$ and $\BR^m$, respectively. Let $A \in \BR^{n \times n}$, $B \in \BR^{n \times m}$ be constant matrices. We will assume that $A$ is invertible. Our linear system of interest is 
\begin{equation}\label{E:linsys}
\frac{d x}{dt} = A x + B u; \qquad x(0)=x_0 \in \BR^n,
\end{equation}
where $x(t):[0,\infty) \to \BR^n$ represents the state of our system, and $u(t):[0,\infty) \to \BR^m$ is a control input. With the goal of meeting a certain objective (e.g., asymptotically stabilizing the system, minimizing a cost functional, etc.), we use the feedback control law $u=-Kx$ where $K \in \BR^{m \times n}$ is a suitable matrix, to obtain the closed-loop system $\frac{d x}{dt} = (A-BK) x$, $x(0)=x_0 \in \BR^n$. Of course, this can be equivalently written in integral form as
\begin{equation}\label{E:linsys-closedloop}
x(t)=x_0 + \int_0^t (A-BK)x(s) \thinspace ds,
\end{equation}
and can be easily solved to yield $x(t) = e^{t(A-BK)} x_0$.

Now consider the following sample-and-hold implementation of the above feedback control law using a zero-order hold \cite{YuzGoodwin-book}. We fix $\delta >0$, and assume that samples are taken at uniformly spaced time instants $t_k^\delta \triangleq k \delta$, $k \in \BZ^+$. Thus, at each time $k\delta$, $k \in \BZ^+$, the state $x^\delta(k\delta)$ is measured, the control is computed according to $u^\delta_{k} \triangleq -K x^\delta(k\delta)$, and is held fixed in \eqref{E:linsys} over the time interval $[k\delta,{(k+1)}\delta)$. The state $x^\delta(t)$ is now obtained by successively solving \eqref{E:linsys} over time intervals $(k\delta,(k+1)\delta)$ with $u=u^\delta_{k}$, using as initial condition $x^\delta(k\delta)=x^\delta(k\delta-)$, and concatenating the pieces. Note that the state $x^\delta(t)$ depends on $\delta$; of course, one expects that as $\delta \searrow 0$, $x^\delta(t)$ converges to $x(t)$ solving \eqref{E:linsys-closedloop}.

To summarize, if we set $x^\delta(0-) \triangleq x^\delta(0)=x_0 \in \BR^n$, then the state $x^\delta(t)$ evolves according to 
a \textit{hybrid dynamical system} \cite{GST-HybDynSys-book} whose evolution for $t \in (k\delta,(k+1)\delta)$, $k \in \BZ^+$, is governed by 
\begin{equation}\label{E:hybsys}
\begin{aligned}
\frac{dx^\delta}{dt} &= Ax^\delta + B u^\delta_{k}, \qquad \text{for $t \in (k\delta,(k+1)\delta)$,} \\
x^\delta(k\delta) & = x^\delta(k\delta -), \medspace \qquad u^\delta_{k} \triangleq - Kx^\delta(k\delta -).
\end{aligned}
\end{equation}
The solution to \eqref{E:hybsys} can be explicitly computed. Indeed, using the variation of constants formula \cite[Theorem 5.2]{Hes_LST}, we get 
$x^\delta(t) = \left[e^{(t-k\delta)A}  - \int_{k\delta}^t e^{(t-s)A} BK \thinspace ds\right] x^\delta(k\delta-)$ for $t \in [k\delta,(k+1)\delta)$, $k \in \BZ^+$.
Thus, we have
\begin{equation}\label{E:state-det}
x^\delta(t) = \sum_{k \ge 0} \ind_{[k\delta,(k+1)\delta)} (t) \thinspace \left[e^{(t-k\delta)A}  - \int_{k\delta}^t e^{(t-s)A} BK \thinspace ds\right] x^\delta(k\delta-). 
\end{equation}
We note that while $x^\delta(t)$ is continuous for $t \in [0,\infty)$, it is only piecewise-smooth; indeed, we typically have $\dot{x}^\delta(k\delta-) \neq \dot{x}^\delta(k\delta+)$ for $k \ge 1$.

\begin{rmk}
Although the state $x^\delta(t)$ is continuous for all $t \ge 0$, we have been a bit fussy in denoting the measurement at time $k\delta$ by $x^\delta(k\delta-)$ (as opposed to $x^\delta(k\delta)$) to underscore the fact that the initial condition needed to solve \eqref{E:hybsys} on the subinterval $(k\delta,(k+1)\delta)$ is available from the solution at the end of the previous subinterval $x^\delta(k\delta-)$. These same considerations apply to the solutions of the {\sc sde} considered below. 
\end{rmk}

We would now like to explore the situation where the system described by \eqref{E:hybsys} is subjected to two different (independent) sources of uncertainty. First, we will assume that there is some sensor or observation noise \cite{Zeitouni-TAC1988}, due to which each state measurement (at the discrete times $k\delta$) yields the sum of the true value of the state and a small additive random error. Secondly, we assume that the ``physical" system itself is subjected to 
small white-noise perturbations (as might be attributable to fluctuating external forces). Both these random effects will be assumed to be of size $\eps$, $0 < \eps \ll 1$. Our state variable will now be a continuous-time stochastic process $\{X^{\eps,\delta}_t:t \ge 0\}$ taking values in $\BR^n$.

To make this precise, we start with a complete probability space $(\Omega,\filt,\BP)$ equipped with a filtration $\{\filt_t:t \ge 0\}$ which satisfies the usual conditions \cite{KS91}. We assume further that this setup supports two independent $n$-dimensional Brownian motions $W=\{W_t: t \ge 0\}$ and $V=\{V_t:t \ge 0\}$. As in the case of \eqref{E:hybsys}, we will assume that the state is measured at the uniformly spaced time instants $t_k^\delta=k\delta$, $k \in \BZ^+$. This time, however, each state measurement will yield the sum of the true value $X^{\eps,\delta}_{k\delta-}$ with a small error term $\eps V_{k\delta}$ due to the measurement noise. Taking the control input over the interval $[k\delta,(k+1)\delta)$ to be $U^{\eps,\delta}_{k} \triangleq -K(X^{\eps,\delta}_{k\delta-} + \eps V_{k\delta})$, the dynamics of $X^{\eps,\delta}_t$ over $[k\delta,(k+1)\delta)$ will now be governed by the {\sc sde} $dX^{\eps,\delta}_t = [AX^{\eps,\delta}_t + BU^{\eps,\delta}_{k}]dt + \eps \thinspace dW_t$. As before, we concatenate solutions over successive intervals $[k\delta,(k+1)\delta)$ using as initial condition $X^{\eps,\delta}_{k\delta}=X^{\eps,\delta}_{k\delta-}$. 

Thus, letting $X^{\eps,\delta}_{0-} \triangleq X^{\eps,\delta}_0=x_0 \in \BR^n$, the state $X^{\eps,\delta}_t$ is a stochastic process with continuous sample paths which solves 
\begin{equation}\label{E:sde}
\begin{aligned}
dX^{\eps,\delta}_t &= [AX^{\eps,\delta}_t + BU^{\eps,\delta}_{k}]dt + \eps dW_t \qquad \qquad \text{for $t \in [k\delta,(k+1)\delta)$, $k \in \BZ^+$,}\\
X^{\eps,\delta}_{k\delta}& =X^{\eps,\delta}_{k\delta-},  \quad U^{\eps,\delta}_{k} \triangleq -K(X^{\eps,\delta}_{k\delta-} + \eps V_{k\delta}). 
\end{aligned}
\end{equation}
Note that the solution to \eqref{E:sde} can be explicitly computed. Indeed, for $t \in [k\delta,(k+1)\delta)$, $k \in \BZ^+$, we have
$X^{\eps,\delta}_t = \left[ e^{(t-k\delta)A} - \int_{k\delta}^t e^{(t-s)A} BK \thinspace ds \right] X^{\eps,\delta}_{k\delta-} + \eps \left[ \int_{k\delta}^t e^{(t-s)A} dW_s -  \int_{k\delta}^t e^{(t-s)A} BK V_{k\delta} \thinspace ds \right]$, 
and we thus have that for $t \ge 0$,
\begin{equation}\label{E:state-stoch}
X^{\eps,\delta}_t = \sum_{k \ge 0} \ind_{[k\delta,(k+1)\delta)} (t) \thinspace \left\{ \left[ e^{(t-k\delta)A} - \int_{k\delta}^t e^{(t-s)A} BK \thinspace ds \right] X^{\eps,\delta}_{k\delta-} + \eps \left[ \int_{k\delta}^t e^{(t-s)A} dW_s -  \int_{k\delta}^t e^{(t-s)A} BK V_{k\delta} \thinspace ds \right] \right\}.
\end{equation}
 

If, in \eqref{E:sde}, one fixes $\eps \in (0,1)$ and takes a limit as $\delta \searrow 0$, one expects the limiting dynamics to be governed by the process $\{X^\eps_t:t \ge 0\}$ which solves
\begin{equation}\label{E:closed-loop-sde}
dX^\eps_t = \left[(A-BK)X^\eps_t -\eps BKV_t \right]dt + \eps dW_t, \qquad X^\eps_0=x_0.
\end{equation}
We can now pose our principal questions of interest. In the absence of sampling effects, i.e., in the formal limit when $\delta \searrow 0$ with $\eps \in (0,1)$ fixed, the convergence as $\eps \searrow 0$ of $X^\eps_t$ solving \eqref{E:closed-loop-sde} to $x(t)$ solving \eqref{E:linsys-closedloop} is very straightforward. In fact, even for nonlinear systems $\dot{x}(t)=b(x(t))$, $x(0)=x_0 \in \BR^n$, the convergence to $x(t)$, as $\eps \searrow 0$, of $X^\eps_t$ solving $dX^\eps_t=b(X^\eps_t)dt + \eps \sigma(X^\eps_t)dW_t$, $X^\eps_0=x_0$,  is classical and very well-understood \cite{FW_RPDS}.\footnote{Of course, one needs to impose sufficient regularity conditions on $b$ and $\sigma$.} Our main goal in this paper is to understand how classical limit theorems need to be modified to account for sampling effects. More precisely, we would like to understand how the relative rates at which $\eps$, $\delta \searrow 0$, influence the convergence of $X^{\eps,\delta}_t$ to $x(t)$. 

Following \cite{FreidlinSowers,KS2014}, we will organize our thoughts as follows. We assume that $\delta=\delta_\eps \searrow 0$ as  $\eps \searrow 0$ and  $\lim_{\eps \searrow 0}\delta_\eps/\eps$ exists in $[0,\infty]$. We now identify the following three asymptotic regimes:
\begin{equation*}
\cc \triangleq \lim_{\eps \searrow 0}\delta_\eps/\eps \begin{cases}=0 & \text{Regime 1,}\\ \in (0,\infty) & \text{Regime 2,}\\ = \infty & \text{Regime 3.}\end{cases}
\end{equation*}
For the cases $\cc=0$, $\cc \in (0,\infty)$, set
\begin{equation}\label{E:kappa}
\kap(\eps) \triangleq \left|{\delta}/{\eps}-\cc\right|.
\end{equation}
Of course, $\lim_{\eps \searrow 0} \kap(\eps)=0$. For the case $\cc=\infty$, we will find it more convenient to view $\eps=\eps_\delta \searrow 0$ as $\delta \searrow 0$. 

To state our results, we start by fixing some notation. For $x=(x_1,\dots,x_n) \in \BR^n$, we let $|x| \triangleq \sqrt{\sum_{i=1}^n |x_i|^2}$ be the standard Euclidean norm, and for $A \in \BR^{n\times n}$, we let $|A|$ be the corresponding induced matrix norm. For each $\delta \in (0,1)$, define the map $\pi_\delta:[0,\infty) \to \delta \BZ^+$ by
$\pi_\delta(t) \triangleq \delta \lfloor {t}/{\delta} \rfloor \quad \text{for $t \in [0,\infty)$,}$
where $\lfloor \cdot \rfloor$ denotes the integer floor function. Thus, $\pi_\delta(\cdot)$ is a time-discretization operator which rounds down the continuous time $t \in [0,\infty)$ to the nearest multiple of $\delta$. Note that the function $x^\delta(t)$ given by \eqref{E:state-det} solves the integral equation
\begin{equation}\label{E:ie}
x^\delta(t) = x_0 + \int_0^t \left[A x^\delta(s) - BK x^\delta(\pi_\delta(s))\right] \thinspace ds,
\end{equation}
while $X^{\eps,\delta}_t$ given by \eqref{E:state-stoch} solves the stochastic integral equation
\begin{equation}\label{E:sie}
X^{\eps,\delta}_t = x_0 + \int_0^t \left[A X^{\eps,\delta}_s - BK X^{\eps,\delta}_{\pi_\delta(s)} - \eps BK V_{\pi_\delta(s)}\right] \thinspace ds + \eps \thinspace W_t.
\end{equation}
We now state our first result. 
\begin{thm}\label{T:flln}
Let $x(t)$ and $X_t^{\varepsilon,\delta}$ solve \eqref{E:linsys-closedloop} and \eqref{E:sie} respectively.
Then, for any $\eps>0, \delta>0$ and $T>0$, there exists a positive constant $C$ depending only on $A,B,K$ and $n$ such that 
\begin{equation*}
\BE\left[\sup_{0\le s \le t}|X_s^{\eps,\delta}-x(s)|\right]\le C\left[\varepsilon \sqrt{T}(T+1)+\delta e^{TC}|x_0|\right]e^{Ct}, \qquad \thinspace \forall \thinspace t\in [\thinspace 0,T\thinspace].
\end{equation*} 
\end{thm}
We will prove this result in Section \ref{S:FLLN}.

Next, we would like to understand the typical fluctuations of $X^{\eps,\delta}_t$ about $x(t)$. To proceed, consider the rescaled fluctuation processes 
\begin{equation}\label{E:fluct-processes}
Z^{\eps,\delta}_t \triangleq \frac{X^{\eps,\delta}_t - x(t)}{\eps} \quad \text{for Regimes 1 and 2, and}\qquad 
U^{\eps,\delta}_t \triangleq \frac{X^{\eps,\delta}_t - x(t)}{\delta} \quad \text{for Regime 3}.
\end{equation}
Note that in each case, we are rescaling in terms of the \textit{coarser} parameter, i.e., the parameter which goes to zero more slowly.\footnote{Of course, in Regime 2, one could equivalently rescale with respect to $\delta$.} Of course, we have
$X^{\eps,\delta}_t = x(t) + \eps Z^{\eps,\delta}_t$ in Regimes 1 and 2, and $X^{\eps,\delta}_t = x(t) + \delta U^{\eps,\delta}_t$ in Regime 3.
From \eqref{E:linsys-closedloop} and \eqref{E:sie}, we see that 
\begin{equation}\label{E: cent-limit-eq}
\begin{aligned}
Z_t^{\eps,\delta} &=\int_0^t (A-BK)Z_s^{\eps,\delta}\thinspace ds -BK \int_0^t V_{\pi_\delta(s)}\thinspace ds + BK \int_0^t \frac{X_s^{\varepsilon, \delta}-X_{\pi_\delta(s)}^{\varepsilon, \delta}}{\eps}\thinspace ds+ W_t,\\
U_t^{\eps,\delta}&=\int_0^t (A-BK)U^{\eps,\delta}_s\thinspace ds-\frac{\eps}{\delta}BK\int_0^{t}V_{\pi_{\delta}(s)}\thinspace ds+BK\int_0^t\frac{ X^{\eps,\delta}_s-X_{\pi_{\delta(s)}}^{\eps,\delta}}{\delta}\thinspace ds+\frac{\eps}{\delta} W_t.
\end{aligned}
\end{equation}
Noting that the dynamics of $Z_t^{\eps,\delta}$ and $U_t^{\eps,\delta}$ involve the state $X_t^{\eps,\delta}$ as well as the small parameters $\eps,\delta$, it is now natural to ask whether, in the limit as $\eps,\delta \searrow 0$, $Z_t^{\eps,\delta}$ and $U_t^{\eps,\delta}$ can be replaced by 
\textit{effective} (i.e., independent of $\eps$, $\delta$, $X^{\eps,\delta}_t$) fluctuation processes $Z_t$, $U_t$ such that
\begin{equation*}
X^{\eps,\delta}_t = x(t) + \eps Z_t + o(\eps) \quad \text{in Regimes 1, 2, and} \quad X^{\eps,\delta}_t = x(t) + \delta U_t + o(\delta) \quad \text{in Regime 3,}
\end{equation*}
and if yes, what estimates can be obtained on the remainder. 

Our main result in this regard for Regimes 1 and 2 is the following. In these regimes, since $\lim_{\eps \searrow 0}\delta_\eps/\eps = \cc \in [0,\infty)$, there exists $\eps_0 \in (0,1)$ such that
\begin{equation}\label{E:eps0}
\left|{\delta_{\eps}}/{\eps}-\cc\right|<1 \quad \text{whenever} \quad 0<\eps<\eps_0. \quad \text{In particular, for $0<\eps<\eps_0$, we have $\delta_{\eps}<(\cc+1)\eps$.}
\end{equation}

\begin{thm}\label{T:fluctuations-R-1-2}
Let $x(t)$ and $X_t^{\varepsilon,\delta}$ solve \eqref{E:linsys-closedloop} and \eqref{E:sie} respectively. Suppose that we are in Regime $i \in \{1,2\}$, i.e., $\lim_{\eps \searrow 0}\delta_\eps/\eps = \cc \in [0,\infty)$. 
Let $Z=\{Z_t: t \ge 0\}$ be the unique strong solution of 
\begin{equation}\label{E:lim-fluct-R-1-2}
Z_t=\int_0^t (A-BK) Z_s\thinspace ds-BK \int_0^t V_s\thinspace ds + \frac{\cc}{2} BK \int_0^t (A-BK) x(s) \thinspace ds +W_t.
\end{equation}
Then, there exists $\eps_0 \in (0,1)$ such that for any $T>0$, $0<\eps<\eps_0$, we have
\begin{equation}
\BE\left[\sup_{0 \le t \le T} |X^{\eps,\delta}_t - x(t) - \eps Z_t|\right] \le \eps
(\cc+1)^2 K_{\ref{T:fluctuations-R-1-2}}e^{K_{\ref{T:fluctuations-R-1-2}}T} \left[\sqrt{\eps}(|x_0| +1 + T) + \kap(\eps)|x_0|\right],
\end{equation}
where $K_{\ref{T:fluctuations-R-1-2}}$ is some positive constant which depends only on $A,B,K$ and $n$, and $\kap(\eps) \searrow 0$ is as in \eqref{E:kappa}.
\end{thm}

\begin{rmk}
The process $Z_t$ is thus obtained by formally taking limits as $\eps, \delta \searrow 0$ in the first equation in \eqref{E: cent-limit-eq}, while replacing $(1/{\eps})\int_0^t \left({X_s^{\varepsilon, \delta}-X_{\pi_\delta(s)}^{\varepsilon, \delta}}\right) \thinspace ds$ by the effective drift term $({\cc}/{2}) \int_0^t (A-BK) x(s) \thinspace ds$, which captures the cumulative effect of fast sampling. Note that the latter \textit{does} involve the zeroth-order behavior given by $x(t)$; further, it vanishes in Regime 1 where $\delta \ll \eps$. We also note that Theorem \ref{T:fluctuations-R-1-2} enables us to approximate, in a strong (pathwise) sense, the non-Markovian process $X^{\eps,\delta}_t$ by the time-inhomogeneous Markov process $x(t)+\eps Z_t$, and provides estimates for the ensuing error.
\end{rmk}

\begin{rmk}
Recall that if $\{Y^\eps\}_{\eps \in (0,1)}$ and $Y$ are random variables taking values in a metric space $S$, then we say that \textit{$Y^\eps$ converges in distribution} to $Y$  as $\eps \searrow 0$, denoted $Y^\eps \Rightarrow Y$, if for every bounded continuous function $f:S \to \BR$, we have $\lim_{\eps \searrow 0} \BE[f(Y^\eps)]=\BE[f(Y)]$\cite{ConvProbMeas}. One can easily show that in Regime $i \in \{1,2\}$, i.e., $\lim_{\eps \searrow 0}\delta_\eps/\eps = \cc \in [0,\infty)$ and for any $T>0,~ Z^{\eps,\delta_\eps} \Rightarrow Z$ in $C([0,T];\BR^n)$ as $\eps \searrow 0$; here, the space $C([0,T];\BR^n)$ is given the metric induced by the sup norm.
\end{rmk}

We next state the corresponding result for Regime 3 ($\cc=\infty$). In this regime, since
\begin{equation}\label{E:kappas}
\text{$\tilde\kap(\delta) \triangleq {\eps}/{\delta} \searrow 0$ as $\delta \searrow 0$,  there exists $\delta_0\in(0,1)$ such that whenever $0<\delta<\delta_0$, we have $\eps<\delta$.}  
\end{equation}
As will be seen below, the first-order correction $U_t$ in this case to $x(t)$ is deterministic.

\begin{thm}\label{T:fluctuation_reg_3}
Let $x(t)$ and $X_t^{\varepsilon,\delta}$ solve \eqref{E:linsys-closedloop} and \eqref{E:sie} respectively. Suppose that we are in Regime 3, i.e., $\lim_{\eps \searrow 0}\delta_\eps/\eps =\infty$. Let $U=\{U_t: t \ge 0\}$ be the unique solution of 
\begin{equation}\label{E:lim-fluct_reg3}
U_t=\int_0^t (A-BK) U_s\thinspace ds+\frac{1}{2}BK \int_0^t(A-BK)x(s)\thinspace ds.
\end{equation}
Then, for any $T>0$, and $0<\delta<\delta_0$ with $\delta_0$ as in \eqref{E:kappas}, we have 
\begin{equation}
\BE\left[\sup_{0 \le t \le T} |X^{\eps,\delta}_t - x(t) - \delta U_t|\right] \le \delta\left[\tilde\kap(\delta)+\sqrt{\delta}(1+|x_0|)\right]K_{\ref{T:fluctuation_reg_3}}e^{K_{\ref{T:fluctuation_reg_3}}T}, 
\end{equation}
where $K_{\ref{T:fluctuation_reg_3}}$ is some positive constant depending only on $A,B,K$ and $n$ and $\tilde\kap(\delta) \searrow 0$ is as in \eqref{E:kappas}.
\end{thm}
Since the proof of Theorem \ref{T:fluctuation_reg_3} very closely parallels that of Theorem \ref{T:fluctuations-R-1-2}, we omit the full details for the sake of brevity. In brief, the main part of the proof consists of showing the convergence, as $\delta \searrow 0,$  of $(1/{\delta})\int_0^t ({X_s^{\varepsilon, \delta}-X_{\pi_\delta(s)}^{\varepsilon, \delta}})\thinspace ds$ to $({1}/{2})\int_0^t(A-BK)x(s)\thinspace ds$ in a suitable sense.


Before proceeding with the proofs of the above results, we make a simple observation which will be used repeatedly without explicit mention. Let $C([0,\infty);\BR^n)$ denote the space of all continuous functions taking $[0,\infty)$ into $\BR^n$. Then, for $t \ge 0$, $y,z\in C([0,\infty); \BR^n)$, the triangle inequality yields
\begin{equation}\label{E:lipschitz}
\left|\left[Ay(t)-BKy(\pi_\delta(t))\right]-\left[Az(t)-BKz(\pi_\delta(t))\right]\right|\leq (|A|+|BK|)\sup_{0\leq s \leq t}|y(s)-z(s)|.
\end{equation}

\section{Limiting Mean Behavior}\label{S:FLLN}
Here, we prove Theorem \ref{T:flln}. 


\begin{proof}[Proof of Theorem \ref{T:flln}]
Since, by the triangle inequality, we have
\begin{equation}\label{E:triangle}
\BE\left[\sup_{0 \le s \le t} |X^{\eps,\delta}_s - x(s)|\right] \le \BE\left[\sup_{0 \le s \le t} |X^{\eps,\delta}_s - x^\delta(s)|\right] + \sup_{0 \le s \le t} |x^\delta(s) - x(s)|,
\end{equation}
we can estimate the two terms on the right individually, and then put the pieces together.
Let $t \in [0,T]$.
Using \eqref{E:ie}, \eqref{E:sie} and \eqref{E:lipschitz}, we easily get 
$|X_t^{\varepsilon,\delta}-x^{\delta}(t)|
\le \int_0^t(|A|+|BK|)\sup_{0\leq r\leq s}|X_r^{\varepsilon,\delta}-x^{\delta}(r)|\thinspace ds+\varepsilon |BK|\int_0^t |V_{\pi_\delta(s)}| ds +\varepsilon \sup_{0\le s \le t}|W_s|.$

Since the right-hand side is non-decreasing in $t$, we set $Y_t^{\varepsilon,\delta} \triangleq \sup_{0\le r \le t}|X_r^{\varepsilon,\delta}-x^{\delta}(r)|$, $\|W\|_t^{*} \triangleq \sup_{0\leq s\leq t}|W_s| = \sup_{0\leq s\leq t} \sqrt{\sum_{i=1}^n (W_t^i)^2}$, and obtain 
\begin{equation}\label{E:Z_t}
Y_t^{\varepsilon,\delta}\le \int_0^t(|A|+|BK|)Y_s^{\varepsilon,\delta}\thinspace ds+\varepsilon |BK|\int_0^t |V_{\pi_\delta(s)}| ds+\varepsilon \|W\|_t^{*}.
\end{equation}
Noting that $\sum_{i=1}^n\langle W^{i}\rangle_t=nt$, it follows from the Burkholder-Davis-Gundy inequality for vector-valued martingales \cite[Problem 3.3.29]{KS91} that for any $m>0$, there exist universal positive constants $\lambda_m, \Lambda_m$ such that for all $t \ge 0$,
\begin{equation}\label{E:BDG}
\lambda_m(nt)^m\le \BE[(\|W\|_t^{*})^{2m}]\le \Lambda_m(nt)^m.
\end{equation}
Next, we use the fact that $\BE[|V_{s}|] \le \sqrt{\BE[|V_{s}|^2]} = \sqrt{ns}$ for any $s \ge 0$ to get 
\begin{equation}\label{E:G-delta}
\BE\left[\int_0^t |V_{\pi_\delta(s)}| ds\right]\le \delta \sum_{k=0}^{\left\lfloor t/\delta\right\rfloor}\BE[|V_{k\delta}|] \le \delta \sqrt{nt} \left(\left\lfloor\frac{t}{\delta}\right\rfloor+1\right) \le \sqrt{nt}(t+1).
\end{equation}
Taking expectations in \eqref{E:Z_t}, and using \eqref{E:BDG}, \eqref{E:G-delta}, we get\\ 
$\BE Y_t^{\varepsilon,\delta} \le (|A|+|BK|)\int_0^t\BE Y_s^{\varepsilon,\delta}\thinspace ds+\varepsilon \left(|BK|\sqrt{nT}(T+1)+\Lambda_{1/2}\sqrt{nT}\right)$.
Letting $C_1(n) \triangleq 2(|BK|+\Lambda_{1/2})\sqrt{n}$, we use Gronwall's inequality to get
\begin{equation}\label{E:exp-Stoch-ineq}
\BE\left[\sup_{0 \le s \le t} |X^{\eps,\delta}_s - x^\delta(s)|\right]\le \varepsilon C_1(n) (T^{3/2} + T^{1/2})e^{(|A|+|BK|)t} \qquad \text{ for $t\in [0,T]$}.
\end{equation}
Now, we estimate $\sup_{0\le s\le t}|x^{\delta}(s)-x(s)|$. From \eqref{E:linsys-closedloop} and \eqref{E:ie}, we easily get
$x^{\delta}(t)-x(t)=\int_0^t A(x^{\delta}(s)-x(s))\thinspace ds-\int_0^t BK \left[x^{\delta}\left(\pi_\delta(s) \right)-x\left(\pi_\delta(s) \right) \right]\thinspace ds-\int_0^t BK \left[x\left(\pi_\delta(s) \right)-x(s)\right]\thinspace ds$. 
Straightforward calculations yield
\begin{equation}\label{E:det-hyb-ineq}
  \sup_{0\le s \le t}|x^{\delta}(s)-x(s)|\le (|A|+|BK|)\int_0^t \sup_{0\le r \le s}|x^{\delta}(r)-x(r)|\thinspace ds+ |BK|\int_0^t \left|x\left(\pi_\delta(s)\right) - x(s)\right|\thinspace ds.
\end{equation}
Noting that 
$x\left(\pi_\delta(s)\right)=e^{\pi_\delta(s)(A-BK)}x_0$, we get
$x(s)-x\left(\pi_\delta(s)\right)
=e^{\pi_\delta(s)(A-BK)}\{e^{h(A-BK)}-I\}x_0$ where $h=s-\pi_\delta(s) \in [0,\delta)$ and $I$ denotes the $n \times n$ identity matrix. 
Since 
$\frac{d}{dt}(e^{t(A-BK)})=(A-BK)e^{t(A-BK)}=e^{t(A-BK)}(A-BK)$, we have $(e^{h(A-BK)}-I)=\int_0^h(A-BK)e^{s(A-BK)}\thinspace ds$ for any $h>0$.
Therefore, 
$\left|x(s)-x\left(\pi_\delta(s)\right)\right| \le e^{\pi_\delta(s)|A-BK|}|e^{h(A-BK)}-I||x_0| \le e^{s|A-BK|}\left\{\int_0^h|A-BK|e^{s|A-BK|}\thinspace ds\right\}|x_0| \le \delta e^{|A-BK|} |A-BK| e^{s|A-BK|} |x_0|$, where we have used the fact that $0 \le h < \delta <1$. Letting $C_2 \triangleq |BK|e^{|A-BK|}$, we get  
\begin{equation*}
|BK|\int_0^t \left|x(s)-x\left(\pi_\delta(s)\right)\right|\thinspace ds \le \delta C_2\left(e^{C_2 T}-1\right)|x_0|. \text{ Now, we use Gronwall's inequality in \eqref{E:det-hyb-ineq} to get}
\end{equation*}
 
\begin{equation}\label{E:exp-det-hyb-ineq}
\sup_{0\le s \le t}|x^{\delta}(s)-x(s)|\le \delta C_2 (e^{T\thinspace C_2}-1)|x_0| e^{(|A|+|BK|)t},\qquad 0\le t\le T.
\end{equation}
If we now let $C(n) \triangleq \max\{C_1(n),C_2,|A|+|BK|\}$, then using the estimates \eqref{E:exp-Stoch-ineq} and \eqref{E:exp-det-hyb-ineq} in \eqref{E:triangle}, we get the stated claim.
\end{proof}

\section{Analysis of fluctuations: Regimes 1 and 2}\label{S:Fluct-R12}
In this section, we prove Theorem \ref{T:fluctuations-R-1-2}. As is evident from equations \eqref{E: cent-limit-eq} and  \eqref{E:lim-fluct-R-1-2}, the central calculations in the proof of Theorem \ref{T:fluctuations-R-1-2} involve showing that, in a suitable sense, we have 
\begin{equation}\label{E:ell}
\lim_{\substack{\eps,\delta \searrow 0\\ \delta/\eps \to \cc}} (1/{\eps}) \int_0^t ({X_s^{\varepsilon, \delta}-X_{\pi_\delta(s)}^{\varepsilon, \delta}})\thinspace ds =\ell(t) \quad \text{where} \quad 
\ell(t) \triangleq ({\cc}/{2}) \int_0^t (A-BK) x(s) \thinspace ds = ({\cc}/{2}) \int_0^t \dot{x}(s) \thinspace ds. 
\end{equation} 
This section is organized as follows. We start with Proposition \ref{P:key-estimate-R-1-2}, which is the key to proving Theorem \ref{T:fluctuations-R-1-2}. Indeed, the estimates in Proposition \ref{P:key-estimate-R-1-2} quantify the error in replacing $(1/{\eps})\int_0^t ({X_s^{\varepsilon, \delta}-X_{\pi_\delta(s)}^{\varepsilon, \delta}})\thinspace ds$ by $\ell(t)$, and codify the precise sense in which \eqref{E:ell} holds. To build up to the proof of Proposition \ref{P:key-estimate-R-1-2}, we next work through a series of lemmas. 
After putting together the proof of Proposition \ref{P:key-estimate-R-1-2}, we close out the section with the proof of  
Theorem \ref{T:fluctuations-R-1-2}.

\begin{pro}\label{P:key-estimate-R-1-2}
Suppose that we are in Regime $i \in \{1,2\}$, i.e., $\lim_{\eps \searrow 0}\delta_\eps/\eps = \cc \in [0,\infty)$. Recall $\eps_0 \in (0,1)$ defined in \eqref{E:eps0}.
Then, there exists a constant $K_{\ref{P:key-estimate-R-1-2}}>0$ depending only on $A,B,K$ and $n$ such that for any $T>0$, $0<\eps<\eps_0$, we have
\begin{equation}\label{E:key-estimate-R-1-2}
\BE \left[\sup_{0 \le t \le T} \left|\int_0^t \frac{{X_s^{\varepsilon, \delta}-X_{\pi_\delta(s)}^{\varepsilon, \delta}}}{\eps}\thinspace ds - 
\ell(t) \right|\right] \le (\cc+1)^2 K_{\ref{P:key-estimate-R-1-2}} e^{K_{\ref{P:key-estimate-R-1-2}} T} \left[\sqrt{\eps}(|x_0| +1 + T) + \kap(\eps)|x_0|\right],
\end{equation}
where $\kap(\eps)$ defined in \eqref{E:kappa} satisfies $\lim_{\eps \searrow 0}\kap(\eps)=0$.
\end{pro}

\noindent
To start working our way up to the proof of Proposition \ref{P:key-estimate-R-1-2}, we next state and prove Lemma \ref{L:errors}, which uses \eqref{E:state-stoch} to explicitly compute $(1/{\eps})\int_0^t ({X_s^{\varepsilon, \delta}-X_{\pi_\delta(s)}^{\varepsilon, \delta}})\thinspace ds$; this enables us to express $(1/{\eps})\int_0^t ({X_s^{\varepsilon, \delta}-X_{\pi_\delta(s)}^{\varepsilon, \delta}})\thinspace ds - \ell(t)$ as a sum of three terms, which are subsequently estimated in Lemmas \ref{L:errors} through \ref{L:L_Stoch_2-R-1-2}.  
To simplify some of the notation, we let $M=\{M_t: 0 \le t < \infty\}$ be the process defined by
\begin{equation}\label{E:mart}
M_t \triangleq \int_0^t e^{-sA}\thinspace dW_s =e^{-tA}W_t+\int_0^t e^{-sA}AW_s \thinspace ds,
\end{equation}
where the latter follows from the integration by parts formula.

\begin{lem}\label{L:errors}
For $\eps,\delta \in (0,1)$, $t \ge 0$, we have $(1/{\eps})\int_0^t ({X_s^{\varepsilon, \delta}-X_{\pi_\delta(s)}^{\varepsilon, \delta}})\thinspace ds = \sum_{i=1}^3 L^{\eps,\delta}_i(t)$ where
\begin{equation}\label{E:errors}
\begin{aligned}
L^{\eps,\delta}_1(t) &= (1/{\eps})\int_0^t \left[{e^{{(s-\pi_\delta(s))}A}-I}\right] \left[I-A^{-1}BK\right]X_{\pi_\delta(s)}^{\eps,\delta} \thinspace ds,\\
L^{\eps,\delta}_2(t) &= \int_0^t e^{sA}\left(M_s-M_{\pi_\delta(s)}\right) \thinspace ds, \\
L^{\eps,\delta}_3(t) &= -\int_0^t \left(e^{(s-\pi_\delta(s))A} - I \right)A^{-1}BKV_{\pi_\delta(s)} \thinspace ds.
\end{aligned}
\end{equation}
\end{lem}

\begin{proof}[Proof of Lemma \ref{L:errors}]
Recalling \eqref{E:state-stoch}, it follows that for $t\in [k\delta,(k+1)\delta),~k\in \BZ^+,$ we have
$X^{\eps,\delta}_t - X^{\eps,\delta}_{k\delta-}=\\ \left[ e^{(t-k\delta)A} - e^{tA}\int_{k\delta}^t e^{-sA} BK \thinspace ds - I\right] X^{\eps,\delta}_{k\delta-} + \eps e^{tA}\int_{k\delta}^t e^{-sA} dW_s - \eps e^{tA}\int_{k\delta}^t e^{-sA} BK V_{k\delta} \thinspace ds.$
Since $A$ is invertible, we have $\frac{d}{dt}(e^{-tA}A^{-1} )=-e^{-tA}$. This implies that $\int_{k\delta}^t e^{-sA} ds = -\left(e^{-tA} - e^{-k\delta A}\right) A^{-1}$. We now easily compute that 
$X_t^{\varepsilon,\delta}-X_{k\delta-}^{\eps,\delta} =\left[e^{(t-k\delta)A}-I\right] \left[ I - A^{-1}BK \right] X_{k\delta-}^{\eps,\delta}+\varepsilon e^{tA}\int_{k\delta}^t e^{-sA}\thinspace dW_s-\varepsilon e^{tA} \int_{k\delta}^t e^{-sA}BK V_{k\delta}\thinspace ds.$
Thus, for $t \ge 0$, we have

\begin{equation}\label{E:sol_diff}
X_t^{\eps,\delta}-X_{\pi_\delta(t)}^{\varepsilon,\delta}=\left[e^{{(t-\pi_\delta(t))}A}-I\right] \left[I-A^{-1}BK\right]X_{\pi_\delta(t)}^{\eps,\delta}+\varepsilon e^{tA}\left(M_t-M_{\pi_\delta(t)}\right)-\eps \left(e^{(t-\pi_\delta(t))A} - I \right)A^{-1}BKV_{\pi_\delta(t)},
\end{equation}
where $M_t$ is as in \eqref{E:mart}. Recalling \eqref{E: cent-limit-eq}, the claim now follows.
\end{proof}

Evidently, one can estimate $(1/{\eps})\int_0^t ({X_s^{\varepsilon, \delta}-X_{\pi_\delta(s)}^{\varepsilon, \delta}})\thinspace ds - \ell(t)$ by separately estimating $|L^{\eps,\delta}_1(t)-\ell(t)|$, $|L^{\eps,\delta}_2(t)|$, $|L^{\eps,\delta}_3(t)|$, and then putting the pieces together. This will be accomplished in Lemmas \ref{L:L_Stoch_1-R-1-2}, \ref{L:L_det-R-1-2}, and \ref{L:L_Stoch_2-R-1-2}, respectively. But first, we state and prove an auxiliary estimate in Lemma \ref{L:J_estimate} which is central to the proof of Lemma \ref{L:L_Stoch_1-R-1-2}.

\begin{lem}\label{L:J_estimate}
Let
\begin{equation}\label{E:J-delta-def} 
J^\delta(t) \triangleq \int_0^t\left(\frac{e^{(s-\pi_{\delta}(s))A}-I}{\delta}-\frac{1}{2}A\right)(I-A^{-1}BK)\thinspace x(\pi_{\delta}(s)) \thinspace ds.
\end{equation}
There exists a constant $C_{\ref{L:J_estimate}}>0$ depending on $A$, $B$, $K$ such that 
\begin{equation*}
\sup_{0\le t\le T}|J^\delta(t) |\le \delta C_{\ref{L:J_estimate}} (1+T) e^{C_{\ref{L:J_estimate}}T}.
\end{equation*}
\end{lem}

\begin{proof}
Letting 
$g^\delta(s) \triangleq \left(\frac{e^{sA}-I}{\delta}-\frac{1}{2}A\right)(I-A^{-1}BK),$
it is easily seen that
\begin{equation}\label{E:J_Estimate}
J^\delta(t) = I_1+I_2, \quad \text{where} \quad I_1 \triangleq  \left(\int_0^{\delta} g^\delta(s) \thinspace ds\right) \sum_{k=0}^{\lfloor\frac{t}{\delta}\rfloor-1} x(k\delta) \thinspace , \quad I_2 \triangleq \left(\int_{0}^{t-\pi_\delta(t)}g^\delta(s) \thinspace \thinspace ds \right) x\left(\pi_\delta(t) \right).
\end{equation}
For $0<r \le \delta$, a direct calculation yields 
\begin{equation}\label{E:gdelta-intformula}
\begin{aligned}
\int_0^r g^\delta(s) \thinspace ds &= \frac{1}{\delta} A^{-1} \left( e^{rA} - I - rA - \frac{1}{2}\delta r A^2 \right) \left(I-A^{-1}BK\right)\\
&= \frac{1}{\delta} A^{-1} \left(A^3 \int_0^r \int_0^s \int_0^v e^{pA} \thinspace dp \thinspace dv \thinspace ds +\frac{1}{2}r(r-\delta)A^2\right) \left(I-A^{-1}BK\right)\\
\left|\int_0^r g^{\delta}(s)\thinspace ds\right| &\le \frac{1}{\delta}\left(|A|^2 \int_0^r \int_0^s \int_0^v e^{r|A|} \thinspace dp \thinspace dv \thinspace ds +\frac{1}{2}r|r-\delta||A| \right)|I-A^{-1}BK|\\
&=\frac{1}{\delta}\left(|A|^2 e^{r|A|}\frac{r^3}{6}+\frac{1}{2}r|r-\delta||A|\right)|I-A^{-1}BK|\\
&\le \left(\frac{r^2}{6}|A|^2 e^{|A|}+\frac{1}{2}|r-\delta||A|\right)|I-A^{-1}BK|,
\end{aligned}
\end{equation}
where the last inequality follows from $r\le \delta<1.$ In case of $I_1$ when $r=\delta,$ we get $\left|\int_0^{\delta} g^{\delta}(s)\thinspace ds\right|\le \frac{\delta^2}{6}|A|^2e^{|A|}|I-A^{-1}BK|$ and in case of $I_2$ when $r=t-\pi_\delta(t), $ we get $|\int_0^{t-\pi_\delta(t)}g^{\delta}(s)\thinspace ds|\le \left(\frac{\delta^2}{6}|A|^2e^{|A|}+\frac{1}{2}\delta |A|\right)|I-A^{-1}BK|$.
Therefore, from \eqref{E:J_Estimate}, we get
\begin{align*}
|J^{\delta}(t)|&\le |I_1|+|I_2|\\
&\le \frac{\delta}{6}|A|^2e^{|A|}|I-A^{-1}BK|\left(\delta \left|\sum_{k=0}^{\lfloor\frac{t}{\delta}\rfloor-1} x(k\delta)\right|\right)+\left(\frac{\delta^2}{6}|A|^2e^{|A|}+\frac{1}{2}\delta |A|\right)|I-A^{-1}BK||x(\pi_\delta(t))|\\
&\le \frac{\delta}{6}|A|^2e^{|A|}|I-A^{-1}BK|\int_0^{\delta\lfloor\frac{t}{\delta}\rfloor}|x(\pi_{\delta}(s))|\thinspace ds+\left(\frac{\delta^2}{6}|A|^2e^{|A|}+\frac{1}{2}\delta |A|\right)|I-A^{-1}BK|\sup_{0\le t \le T}|x(t)|\\
&\le \frac{\delta}{6}|A|^2e^{|A|}|I-A^{-1}BK|T \sup_{0 \le t \le T}|x(t)|+ \left(\frac{\delta^2}{6}|A|^2e^{|A|}+\frac{1}{2}\delta |A|\right)|I-A^{-1}BK|\sup_{0\le t \le T}|x(t)|\\
&= \delta \left[\frac{1}{6}|A|e^{|A|}T+\frac{\delta}{6}|A|e^{|A|}+\frac{1}{2}\right]|A||(I-A^{-1}BK)|\sup_{0\le t \le T}|x(t)|.
\end{align*}
Recalling that $x(t)=e^{t(A-BK)}x_0$, the claim easily follows.
\end{proof}


\begin{lem}\label{L:L_Stoch_1-R-1-2}
There exists a constant $D_{\ref{L:L_Stoch_1-R-1-2}}>0$ depending only on $A$, $B$, $K$ and $n$ such that for $0 < \eps < \eps_0$, we have 
\begin{equation}\label{E:L1-estimate}
\BE\left[\sup_{0 \le t \le T} |L^{\eps,\delta}_1(t)-\ell(t)|\right] \le (\cc+1)^2 D_{\ref{L:L_Stoch_1-R-1-2}} e^{D_{\ref{L:L_Stoch_1-R-1-2}}T} \left[\eps \left\{|x_0| + 1 + T\right\} + \kap(\eps)|x_0|\right].
\end{equation}
\end{lem}

\begin{proof}
Setting $f^{\delta}(s)\triangleq \left[\frac{e^{(s-\pi_{\delta}(s))A}-I}{\delta}\right]$, we have
\begin{equation*}
L_1^{\eps,\delta}(t)-\ell(t) =
({\delta}/{\eps}) \int_0^t f^\delta(s)A^{-1}(A-BK)X_{\pi_{\delta}(s)}^{\eps,\delta}\thinspace ds-({\cc}/{2})\int_0^t (A-BK)x(s)\thinspace ds = \sum_{i=1}^4 \sfG_i^{\eps,\delta}(t),
\end{equation*}
where
\begin{equation}\label{E:V1-V4}
\begin{aligned}
\sfG_1^{\eps,\delta}(t) &\triangleq \frac{\delta}{\eps}\int_0^t f^{\delta}(s)A^{-1}(A-BK)\left(X_{\pi_{\delta}(s)}^{\eps,\delta}-x(\pi_{\delta}(s)) \right) \thinspace ds,\\
\sfG_2^{\eps,\delta}(t) &\triangleq \frac{\delta}{\eps}\int_0^t \left(f^{\delta}(s)-\frac{1}{2}A\right)A^{-1}(A-BK) x(\pi_{\delta}(s))\thinspace ds,\\
\sfG_3^{\eps,\delta}(t) &\triangleq \frac{1}{2}\frac{\delta}{\eps}\int_0^t(A-BK)\left(x(\pi_\delta(s))-x(s)\right) \thinspace ds, \qquad 
\sfG_4^{\eps,\delta}(t) \triangleq \frac{1}{2}\left(\frac{\delta}{\eps}-\cc\right)\int_0^t(A-BK)x(s)\thinspace ds. 
\end{aligned}
\end{equation}
Therefore,
\begin{equation}\label{E:L1-error-terms}
\BE\left[\sup_{0 \le t \le T} |L^{\eps,\delta}_1(t)-\ell(t)|\right]\le \BE[\sup_{0\le t \le T}|\sfG_1^{\eps,\delta}(t)|]+\sum_{i=2}^4\sup_{0\le t\le T}|\sfG_i^{\eps,\delta}(t)|.
\end{equation}
Since $f^\delta(s)=\delta^{-1}\int_0^{s-\pi_{\delta}(s)} A e^{rA} \thinspace dr$, we have $|f^{\delta}(s)| 
\le \delta^{-1}|A|\int_0^{s-\pi_{\delta}(s)}e^{\delta|A|}\thinspace dr
\le |A|e^{\delta |A|}$. We now easily get
$|\sfG_1^{\eps,\delta}(t)| \le \frac{\delta}{\eps}|A|e^{\delta |A|}|A^{-1}||A-BK|\int_0^t \sup_{0\le u \le s}\left|X_{u}^{\eps,\delta}-x(u)\right|\thinspace ds$. Hence, using Theorem \ref{T:flln}, we see that there exists a constant $C_{\ref{L:L_Stoch_1-R-1-2}}>0$ depending only on $A$, $B$, $K$ and $n$ such that 
\begin{equation}\label{E:G1-eps-delta}
\begin{aligned}
\BE\left[\sup_{0\le t \le T}|\sfG_1^{\eps,\delta}(t)|\right]&\le \frac{\delta}{\eps}|A|e^{|A|}|A^{-1}||A-BK|\int_0^T \BE\left[\sup_{0\le u \le s}\left|X_{u}^{\eps,\delta}-x(u)\right|\right]\thinspace ds\\
&\le \frac{\delta}{\eps}|A|e^{|A|}|A^{-1}||A-BK|C\left[\eps \sqrt{T}(T+1)+\delta e^{CT}|x_0|\right]e^{CT}\\
& \le \frac{\delta}{\eps} C_{\ref{L:L_Stoch_1-R-1-2}}\left[\eps \sqrt{T}+\delta|x_0|\right]e^{C_{\ref{L:L_Stoch_1-R-1-2}} T}. 
\end{aligned}
\end{equation}
For $0< \eps < \eps_0$, we have $\delta< (\cc+1)\eps$, which implies that 
\begin{equation}\label{E:G1}
\BE\left[\sup_{0\le t \le T}|\sfG_1^{\eps,\delta}(t)|\right] \le \eps (\cc+1) C_{\ref{L:L_Stoch_1-R-1-2}} \left[\sqrt{T} + (\cc+1)|x_0|\right] e^{C_{\ref{L:L_Stoch_1-R-1-2}} T}.
\end{equation}
Noting that $\sfG_2^{\eps,\delta}(t)=\frac{\delta}{\eps}J^{\delta}(t)$, where $J^{\delta}(t)$ is given by \eqref{E:J-delta-def}, it follows from Lemma \ref{L:J_estimate} that for $0<\eps < \eps_0$, we have
\begin{equation}\label{E:G2}
\sup_{0\le t \le T}|\sfG_2^{\eps,\delta}(t)|\le \eps (\cc+1)^2 C_{\ref{L:J_estimate}} (1+T) e^{C_{\ref{L:J_estimate}}T}.
\end{equation}
To estimate $|\sfG_3^{\eps,\delta}(t)|$, we use the fact that $|x(\pi_\delta(s))-x(s)|\le \delta e^{|A-BK|}|A-BK|e^{s|A-BK|}|x_0|$ and easily check that for $0<\eps <\eps_0$, we have 
$\sup_{0\le t \le T}|\sfG_3^{\eps,\delta}(t)| \le \frac{1}{2}\eps (\cc+1)^2 |A-BK| e^{|A-BK|} |x_0| e^{|A-BK|T}$. Hence, there exists a constant $\tilde{C}_{\ref{L:L_Stoch_1-R-1-2}}>0$ depending only on $A$, $B$, $K$ such that 
\begin{equation}\label{E:G3}
\sup_{0\le t \le T}|\sfG_3^{\eps,\delta}(t)| \le \eps(\cc+1)^2 |x_0| \tilde{C}_{\ref{L:L_Stoch_1-R-1-2}} e^{\tilde{C}_{\ref{L:L_Stoch_1-R-1-2}}T}.
\end{equation}
It is easily checked that 
\begin{equation}\label{E:G4}
\sup_{0\le t \le T}|\sfG_4^{\eps,\delta}(t)| \le \frac{1}{2} \kap(\eps) |x_0| e^{|A-BK|T}.
\end{equation}
Putting together equations \eqref{E:L1-error-terms} through \eqref{E:G4}, some simple calculations yield \eqref{E:L1-estimate}.
\end{proof}

\begin{lem}\label{L:L_det-R-1-2}
For $0 < \eps < \eps_0$, we have $\BE\left[\sup_{0 \le t \le T} |L^{\eps,\delta}_2(t)|\right] \le \sqrt{\eps} (\cc+1){K_{\ref{L:L_det-R-1-2}}} e^{K_{\ref{L:L_det-R-1-2}}T}$; where ${K_{\ref{L:L_det-R-1-2}}}$ is some positive constant depending only on $A$ and $n$.
\end{lem}

\begin{proof}
For any $t \ge 0$, we have
$e^{tA}(M_t-M_{\pi_\delta(t)})=W_t-e^{(t-\pi_\delta(t)) A}W_{\pi_\delta(t)}+e^{tA}\int_{\pi_\delta(t)}^t e^{-sA} AW_s \thinspace ds$.
Adding and subtracting $e^{(t-\pi_\delta(t)) A}W_t$ on the right, we now have the estimate
\begin{align*}
|e^{tA}(M_t-M_{\pi_\delta(t)})| 
& \le |I-e^{(t-\pi_\delta(t))A}|\sup_{0\le s \le t}|W_s|+e^{(t-\pi_\delta(t))|A|}|W_t-W_{\pi_\delta(t)}|+e^{t|A|}\left(\int_{\pi_\delta(t)}^t e^{s|A|}|A|\thinspace ds\right)\sup_{0\le s \le t}|W_s|\\
&\le |I-e^{(t-\pi_\delta(t))A}|\sup_{0\le s \le t}|W_s|+e^{\delta|A|}|W_t-W_{\pi_\delta(t)}|+e^{2t|A|}|A| \delta\sup_{0\le s \le t}|W_s|,
\end{align*}
where we have used the fact that for any $s \in \BR$, $|e^{sA}| \le e^{|s||A|}$. Since 
$|e^{hA}-I|\le \int_0^h|A||e^{sA}|\thinspace ds \le h|A|e^{h|A|} \le \delta|A|e^{\delta|A|} \le \delta|A|e^{|A|}$ for any $h \in [0,\delta)$, we have\\
$|e^{tA}(M_t-M_{\pi_\delta(t)})| \le \delta |A| \left(e^{|A|} + e^{2t|A|}\right) \sup_{0 \le s \le t} |W_s| + e^{|A|} |W_t - W_{\pi_\delta(t)}|.$ 
Recalling \eqref{E:errors}, we get\\
$|L_2^{\eps,\delta}(t)|\le 
 \delta |A| \left(e^{|A|} + e^{2t|A|}\right)\int_0^t \sup_{0 \le r \le s} |W_r| \thinspace ds+ e^{|A|}\int_0^t |W_s - W_{\pi_\delta(s)}|\thinspace ds.$
Taking supremum over $t \in [0,T]$, followed by expectation, we use the fact that $\BE\left[|W_s-W_{\pi_\delta(s)}|\right] \le \sqrt{n\delta}$ to get 
\begin{equation*}
\BE\left[\sup_{0\le t \le T}|L_2^{\eps,\delta}(t)|\right] \le \delta |A| \left(e^{|A|} + e^{2T|A|}\right)\int_0^T \BE\left[\sup_{0 \le r \le s} |W_r|\right] \thinspace ds+ e^{|A|}T\sqrt{n\delta }.
\end{equation*}
Straightforward calculations using the Burkholder-Davis-Gundy inequality now yield the result.
\end{proof}


\begin{lem}\label{L:L_Stoch_2-R-1-2}
For $0 < \eps < \eps_0$, we have $\BE\left[\sup_{0 \le t \le T} |L^{\eps,\delta}_3(t)|\right] \le \eps (\cc+1) K_{\ref{L:L_Stoch_2-R-1-2}}e^{K_{\ref{L:L_Stoch_2-R-1-2}}T}$,  where $K_{\ref{L:L_Stoch_2-R-1-2}}$ is some positive constant depending only on $A,B,K$ and $n$.
\end{lem}

\begin{proof}
Recalling that $|e^{(s-\pi_{\delta}(s))A}-I| \le \delta |A| e^{\delta |A|}$, we get
\begin{equation*}
|L_3^{\eps,\delta}(t)| \le \int_0^t|e^{(s-\pi_{\delta}(s))A}-I||A^{-1}BK|
|V_{\pi_{\delta}(s)}|\thinspace ds 
\le \delta |A|e^{\delta |A|}|A^{-1}BK|\int_0^t\sup_{0\le u\le s}|V_u|\thinspace ds.
\end{equation*}
Taking supremum over $t \in [0,T]$, followed by expectation, we get
\begin{equation*}
\BE\left[\sup_{0\le t \le T}|L_3^{\eps,\delta}(t)|\right] \le \eps(\cc+1) |A|e^{ |A|}|A^{-1}BK|\int_0^T \BE\left[\sup_{0\le u\le s}|V_{u}|\right]\thinspace ds,
\end{equation*}
where we have used the fact that for $0<\eps<\eps_0$, one has $\delta < (\cc+1) \eps$. 
Once again, straightforward calculations using the Burkholder-Davis-Gundy inequalities yield the desired result.
\end{proof}


We now provide the proof of Proposition \ref{P:key-estimate-R-1-2}.
\begin{proof}[Proof of Proposition \ref{P:key-estimate-R-1-2}]
By Lemma \ref{L:errors}, we have $\int_0^t \frac{X_s^{\varepsilon, \delta}-X_{\pi_\delta(s)}^{\varepsilon, \delta}}{\eps}\thinspace ds - \ell(t)=\sum_{i=1}^3 L_i^{\eps,\delta}(t) - \ell(t)$. Straightforward calculations using Lemmas \ref{L:L_Stoch_1-R-1-2}, \ref{L:L_det-R-1-2} and \ref{L:L_Stoch_2-R-1-2} easily yield the result.
\end{proof}

Finally, we prove Theorem \ref{T:fluctuations-R-1-2}. 
\begin{proof}[Proof of Theorem \ref{T:fluctuations-R-1-2}]
To start, we note that 
$X^{\eps,\delta}_t - x(t) - \eps Z_t = \eps \left(Z^{\eps,\delta}_t - Z_t \right)$, where $Z^{\eps,\delta}_t$ and $Z_t$ are given by \eqref{E: cent-limit-eq} and \eqref{E:lim-fluct-R-1-2}, respectively. 
It now easily follows that for any $t \in [0,T]$, we have
\begin{multline*}
\sup_{0 \le s \le t} \left|Z^{\eps,\delta}_s - Z_s\right| \le \int_0^t |A-BK| \sup_{0 \le r \le s}\left|Z^{\eps,\delta}_r-Z_r\right| \thinspace ds + |BK| \int_0^t \left|V_{\pi_\delta(s)}-V_s\right| \thinspace ds\\ + |BK| \sup_{0 \le s \le t}\left|\int_0^s \frac{X_r^{\varepsilon, \delta}-X_{\pi_\delta(r)}^{\varepsilon, \delta}}{\eps}\thinspace dr - 
\ell(s)\right|,
\end{multline*}
where $\ell(t)$ is defined in \eqref{E:ell}. 
Now, by Proposition \ref{P:key-estimate-R-1-2} and the fact that $\BE|V_s-V_{\pi_{\delta}(s)}|\le \sqrt{n(s-\pi_{\delta}(s))}\le \sqrt{n \delta}$, we get
\begin{multline*}
\BE\left[\sup_{0 \le s \le t} \left|Z^{\eps,\delta}_s - Z_s\right|\right]\le \int_0^t|A-BK|\medspace\BE\left[\sup_{0\le r \le s}\left|Z^{\eps,\delta}_r - Z_r\right|\right]\thinspace ds+|BK|T\sqrt{n\delta }\\+|BK| (\cc+1)^2 K_{\ref{P:key-estimate-R-1-2}} e^{K_{\ref{P:key-estimate-R-1-2}} T} \left[\sqrt{\eps}(|x_0| +1 + T) + \kap(\eps)|x_0|\right].
\end{multline*}
Straightforward calculations using the Gronwall's inequality yield the desired result.
\end{proof}

\section{Numerical Example and Simulation}\label{S:Example}
In this section, we illustrate our results in the context of a simple optimal control problem. Let $A,B,Q,R$ be the matrices 
\begin{equation*}
A = \begin{bmatrix}
0 & 1 \\
0.5 & 0
\end{bmatrix}, \quad
B = \begin{bmatrix}
0 \\
 1
\end{bmatrix}, \quad 
Q = \begin{bmatrix}
1 & 0 \\
0 & 1
\end{bmatrix} \quad \text{and} \quad
R = [1],
\end{equation*}
and consider the linear control system $\dot{x}(t)=Ax(t)+Bu(t)$ with $x \in \BR^2$, $u \in \BR$, that can also be viewed as a linearized model of a damp free inverted pendulum \cite{antunes2011volterra}. The infinite time horizon \textit{linear quadratic regulation (LQR)} problem entails finding the control $u(t)$ which minimizes the cost functional $J_{LQR} \triangleq \int_0^{\infty}x(t)^\top Qx(t)+u(t)^\top Ru(t)\thinspace dt$. Since the pairs $(A,B)$ and $(A,Q)$ are stabilizable and detectable respectively, standard results (see, for instance \cite[Theorem 21.2]{Hes_LST}) imply that there exists a symmetric solution $P$ to the algebraic Riccati equation $A^\top P + P A + Q - P B R^{-1} B^\top P=0$ such that $A - B R^{-1} B^\top P$ is a stability matrix. Further, the state feedback control law $u=-Kx$ with $K \triangleq R^{-1} B^\top P$ stabilizes the closed-loop system and is optimal in the sense that it minimizes the cost $J_{LQR}$. For the present example, the matrix $K$ is computed using the MATLAB command $\mathtt{[K]=lqr(A,B,Q,R)}$ and is obtained to be $K = \begin{bmatrix}
1.618  &  2.058
\end{bmatrix}$. 

We now consider a sample-and-hold implementation of this system with periodic sampling at times $k\delta$, $k \in \BZ^+$, and small white noise perturbations of size $\eps$ in the state dynamics; here, $0 < \eps,\delta \ll 1$. For simplicity, we assume that the measurement noise is absent. As in Theorem \ref{T:fluctuations-R-1-2}, we would like to compare the stochastic process $X^{\eps,\delta}(t)$ with $S^\eps(t)\triangleq x(t) + \eps Z(t)$, where $x(t)=e^{t(A-BK)}x(0)$, $X^{\eps,\delta}(t)$ solves the {\sc sde} \eqref{E:sie} with $X^{\eps,\delta}(0)=x(0)$, and $Z(t)$ solves the {\sc sde} \eqref{E:lim-fluct-R-1-2}. The stochastic differential equations for $X^{\eps,\delta}(t),Z(t)$ are solved numerically using the Euler-Maruyama method \cite{Higham2001,KloedenPlaten}. Given that Theorem \ref{T:fluctuations-R-1-2} is a strong (pathwise) approximation result, we use the same Brownian increments to generate the paths of $X^{\eps,\delta}(t)$ and $Z(t)$. 

Figure \ref{F:SamplePath} shows a sample path for $(X_1^{\eps,\delta}(t),X_2^{\eps,\delta}(t))$ along with the corresponding sample path for $(S_1^\eps(t),S_2^\eps(t))$ 
with step size $\Delta t=2^{-5}$, $T=2^3$, $\eps= 2^{-5},~\delta= 2^{-4}$ (and hence $\cc=2$) and with initial conditions $x_1(0)=1.5,~ x_2(0)=0.5,~ Z_1(0)= Z_2(0)=0$. The effect of varying $\eps$ on the error $X^{\eps,\delta}(t)-S^\eps(t)$ is explored in Figure \ref{F:ErrorPlot}. Here, for $\eps=2^{-i}$, $1 \le i \le 7$, and other parameters as earlier, we generate 1000 sample paths of $X^{\eps,\delta}(t)$ and $S^\eps(t)$. Let $e_i$ be the vector $(e_{i,1},e_{i,2})$, where, for $j=1,2$, the quantity $e_{i,j}$ is the mean of $|X^{\eps,\delta}_j(T)-S^\eps_j(T)|$ over the 1000 realizations with $\eps=2^{-i}$. The components $e_{i,1},e_{i,2}$ are plotted against $\eps$ on a $\log_2$-$\log_2$ scale. The plot clearly shows that for $\eps=2^{-i}$, $1 \le i \le 7$, $j=1,2$, the quantity $\log_2 e_{i,j}$ decreases linearly with increasing $i$ (i.e., decreasing $\eps$).

\begin{figure}[h!]
\begin{center}
\includegraphics[height=8.8cm,width=11.8cm]{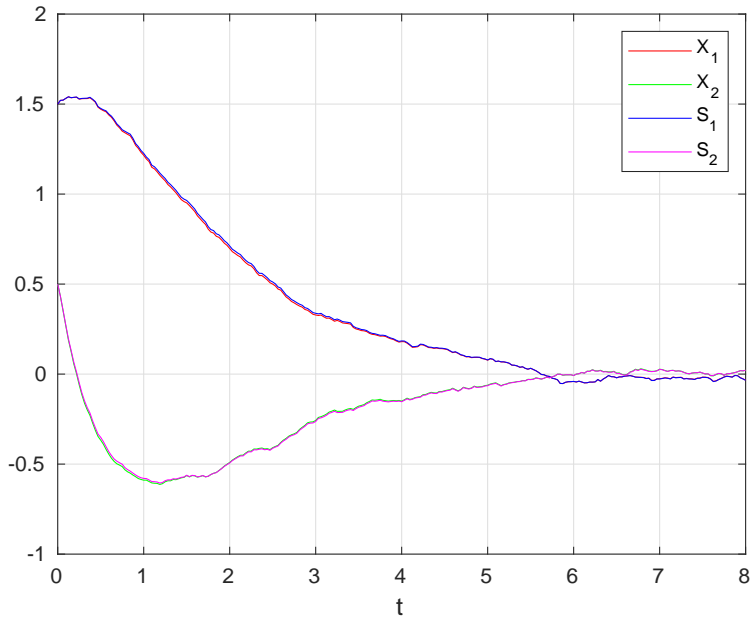}
\caption{Sample paths for the components $X_1\triangleq X_1^{\eps,\delta}(t),~X_2 \triangleq X_2^{\eps,\delta}(t)$ of the {\sc sde} defined by \eqref{E:sie}  and $S_1 \triangleq S_1^{\eps}(t)=x_1(t)+\eps Z_1(t),~S_2 \triangleq S_2^{\eps}(t)=x_2(t)+\eps Z_2(t)$  with $\eps=2^{-5},~\delta=2^{-4},$ and $T=2^3$. Here $Z_1(t)$ and $Z_2(t)$ are the components of $Z(t)$ defined in \eqref{E:lim-fluct-R-1-2}.}\label{F:SamplePath}
\end{center}
\end{figure}

\begin{figure}[h!]
\begin{center}
\begin{equation*}
\includegraphics[height=8.8cm,width=11.8cm]{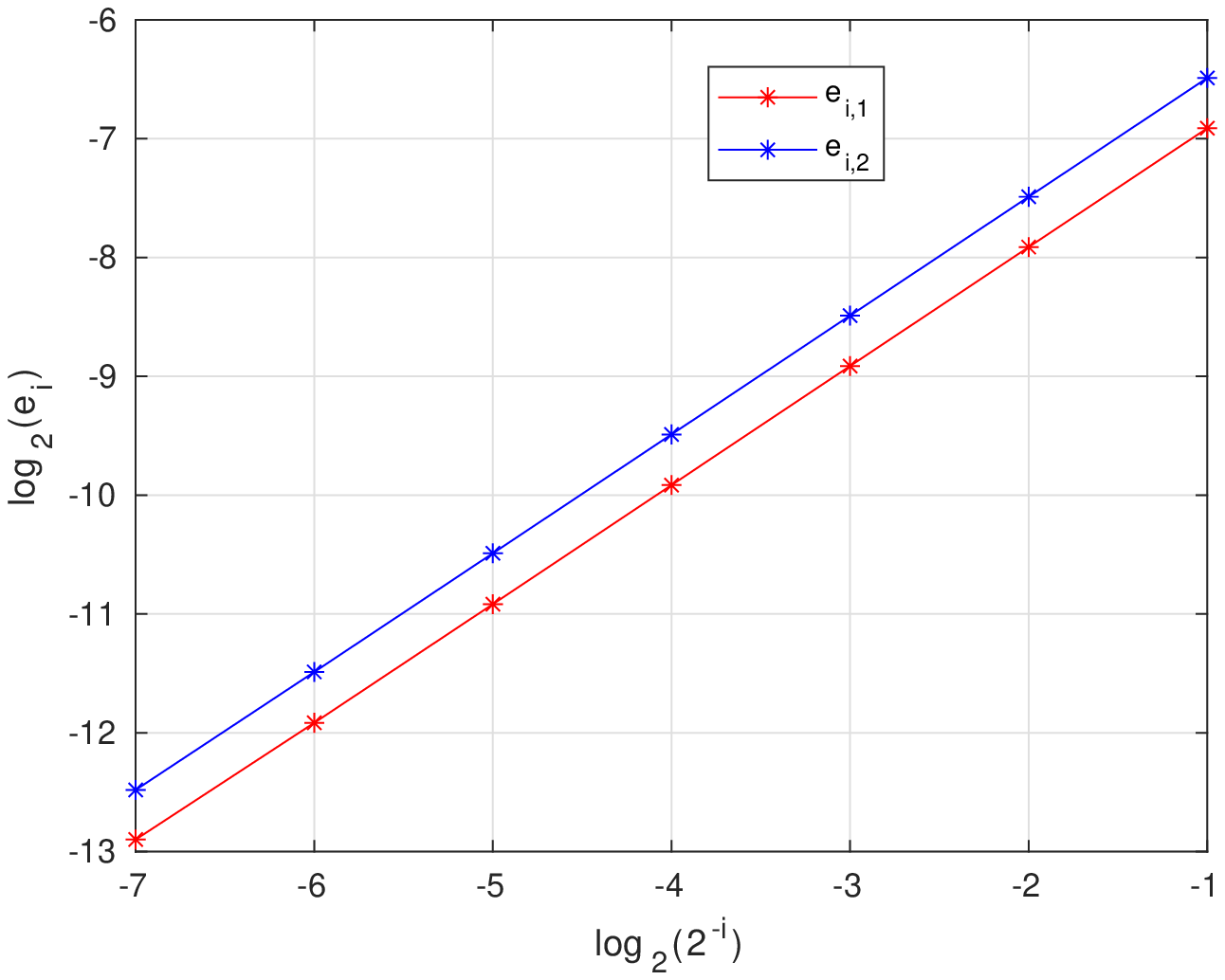}
\end{equation*}
\caption{For $1 \le i \le 7$, let $e_i$ be the vector $(e_{i,1},e_{i,2})$, where, for $j=1,2$, the quantity $e_{i,j}$ is the mean of $|X^{\eps,\delta}_j(T)-S^\eps_j(T)|$ over 1000 sample paths generated by the Euler-Maruyama method, with $T=2^3$, $\delta= 2^{-4}$. On a $\log_2$-$\log_2$ scale, we see that as $\eps\triangleq 2^{-i}$ decreases, the corresponding error $e_{i,j}$ also decreases. 
}\label{F:ErrorPlot}
\end{center}
\end{figure}

\section{Conclusions}\label{S:Conclusions}
In this article, we have studied the combined effect of small Brownian perturbations and fast periodic sampling on the evolution of linear control systems. For the ensuing continuous-time stochastic process indexed by two small parameters, we obtain
 effective ordinary and stochastic differential equations  describing the limiting mean behavior and the typical fluctuations about the mean. The effective fluctuation process is found to vary, depending on the relative rates at which the two small parameters approach zero. The results are illustrated in the context of an infinite time horizon LQR problem. The calculations here suggest several avenues for further exploration, such as looking at nonlinear control systems with periodic sampling subjected to small white noise perturbations, or studying similar questions for the case of fast random, rather than periodic, sampling. 

\bibliographystyle{alpha}
\bibliography{Control}

\end{document}